\documentclass[11pt]{article}
\usepackage[T1]{fontenc}
\usepackage{lmodern}
\usepackage{amsmath,amssymb,amsthm,mathtools}
\usepackage[margin=1in]{geometry}
\usepackage[colorlinks=true,linkcolor=blue,citecolor=blue,urlcolor=blue,hyperfootnotes=false]{hyperref}
\usepackage{microtype}
\usepackage{indentfirst}
\newtheorem{theorem}{Theorem}[section]
\newtheorem{proposition}[theorem]{Proposition}
\newtheorem{lemma}[theorem]{Lemma}
\newtheorem{corollary}[theorem]{Corollary}

\theoremstyle{remark}
\newcommand{\Z}{\mathbb Z}
\newcommand{\Q}{\mathbb Q}
\newcommand{\E}{E}
\newcommand{\wh}{\widehat}
\DeclareMathOperator{\Frac}{Frac}
\title{Laurentness of Svinin's even-order recurrences}
\author{Qiyue Tang}
\date{}
\begin{document}
\maketitle
\begingroup
\renewcommand{\thefootnote}{}
\endgroup
\begin{abstract}
We prove that every term of Svinin's even-order recurrences is a Laurent
polynomial in the initial values over the integer polynomial ring in the
recurrence parameters.
\end{abstract}

\section{Introduction}

Rational recurrences can produce integer sequences because the divisions
introduced at successive steps cancel. The Laurent property expresses a
stronger symbolic form of this cancellation: each term has a denominator
which is a monomial in the initial values. For a sequence $(t_n)_{n\in\Z}$ and fixed parameters
$\alpha_0,\alpha_1$, the Somos--4 recurrence
\begin{equation}\label{eq:somos4}
t_nt_{n+4}=\alpha_0t_{n+2}^2+\alpha_1t_{n+1}t_{n+3}
\end{equation}
has this property over $\Z[\alpha_0,\alpha_1]$, as a case of the two-term
Gale--Robinson theorem of Fomin and Zelevinsky
\cite[Theorem 1.8]{FZ}.

Svinin constructed a family of odd-order homogeneous recurrences and
proved their Laurent property \cite{Svinin}. Section 5.3 of that paper
also defines, for each integer $g\geq1$, an even-order family $R_{2g+2}$ beginning with
\eqref{eq:somos4}, and asks for a proof of its Laurent property.
The subsequent paper \cite[Remark 3.8]{SvininInv} records computational
evidence for the even-order family, whose parameters are
$\alpha_0,\ldots,\alpha_g$. For $g=2$, the recurrence is
\begin{equation}\label{eq:R6}
t_nt_{n+6}=\alpha_0t_{n+2}t_{n+4}
+\alpha_1\frac{t_{n+1}t_{n+4}^2+t_{n+2}^2t_{n+5}}{t_{n+3}}
+\alpha_2t_{n+1}t_{n+5}.
\end{equation}
The denominator $t_{n+3}$ requires cancellation at an intermediate
position of the window $(t_n,\ldots,t_{n+5})$, as well as at $t_n$.

To define the general recurrence, regard the $t_i$, $i\in\Z$, as
commuting indeterminates.
Empty products are $1$ and empty sums are $0$. For integers $n$ and
$s\geq k\geq0$, define the polynomials $B_s^k(n)$ by
\begin{align}
B_s^0(n)&=\prod_{r=1}^{s}t_{n+r},\label{eq:Bzero}\\
B_k^k(n)&=\prod_{r=0}^{k-1}t_{n+2r}t_{n+2r+3}\quad(k\geq1),\label{eq:Bdiag}\\
B_s^k(n)&=t_{n+k+s}B_{s-1}^k(n)
+t_{n+k+s-2}t_{n+k+s+1}B_{s-1}^{k-1}(n)
\quad(s>k\geq1).\label{eq:Brec}
\end{align}
These are the polynomials of \cite[Lemma 1.3]{Svinin}.
For $N=2g+2$, the recurrence $R_N$ is
\begin{equation}\label{eq:RN}
\left(\prod_{r=3}^{N-3}t_{n+r}\right)
\left(t_nt_{n+N}-\alpha_gt_{n+1}t_{n+N-1}
-\alpha_0t_{n+2}t_{n+N-2}\right)
=\sum_{k=1}^{g-1}\alpha_kB_{N-k-3}^k(n+1).
\end{equation}

\begin{theorem}\label{thm:main}
Let $g\geq1$ be an integer and put $N=2g+2$. Take
$\alpha_0,\ldots,\alpha_g,t_0,\ldots,t_{N-1}$ to be algebraically independent
over $\Q$, and generate $(t_n)_{n\in\Z}$ by \eqref{eq:RN} in their rational
function field. Then
\[
t_n\in\Z[\alpha_0,\ldots,\alpha_g]
[t_0^{\pm1},\ldots,t_{N-1}^{\pm1}]
\qquad(n\in\Z).
\]
\end{theorem}

The proof uses Laurent phenomenon (LP) algebras introduced by Lam and
Pylyavskyy \cite{LP}. Their exchange Laurent polynomials are obtained by
dividing the corresponding exchange polynomials by monomials. We construct
a seed whose normalized exchange relation becomes \eqref{eq:RN} after
clearing denominators. After mutation and cyclic relabeling, the seed has
the same form; hence it is period one in the sense of
\cite[Definition 2.4]{ACH}. Independent-set formulas on a path determine
the normalization and give the boundary identities for the mutation.

\paragraph{AI contribution.}
The research proof was generated in a GPT-6-Astra conversation. Codex was
used to reconstruct and check the proof, run finite symbolic checks, and
revise the exposition. Verification comprised AI-assisted proof review
and finite symbolic computation.

\section{Path polynomials and boundary identities}

Let $(x_i)_{i\in\Z}$ be algebraically independent commuting variables.
The parameters $\alpha_0,\ldots,\alpha_g$ are also algebraically independent
over $\Q$ and independent of the $x_i$. We work over
$\Z[\alpha_0,\ldots,\alpha_g]$ with Laurent polynomials in finitely many
of the $x_i$. Assign to
each integer $i$ the Laurent monomial
\begin{equation}\label{eq:u}
u_i=\frac{x_ix_{i+3}}{x_{i+1}x_{i+2}}.
\end{equation}
For integers $p\leq q$ and $k\geq0$, let $\E_k[p,q]$ be the sum of the weights of
$k$-element independent sets of the path with vertices $p,\ldots,q$:
\begin{equation}\label{eq:E}
\E_k[p,q]=
\sum_{\substack{p\leq i_1<\cdots<i_k\leq q\\i_{r+1}-i_r\geq2}}
u_{i_1}\cdots u_{i_k}.
\end{equation}
We put $\E_0[p,q]=1$, including when $p>q$, and put $\E_k[p,q]=0$
when $k<0$ or no such set exists. For example,
$\E_2[p,p+2]=u_pu_{p+2}$: adjacent vertices cannot be chosen together.
Partitioning by the last or the first vertex gives
\begin{align}
\E_k[p,q]&=\E_k[p,q-1]+u_q\E_{k-1}[p,q-2],\label{eq:last}\\
\E_k[p,q]&=\E_k[p+1,q]+u_p\E_{k-1}[p+2,q].\label{eq:first}
\end{align}
These identities hold for $p\leq q$ and all integers $k$ with the stated
zero conventions.

\begin{lemma}\label{lem:B}
For integers $n$ and $s\geq k\geq0$, substitution $t_i=x_i$ gives
\begin{equation}\label{eq:Bpath}
B_s^k(n)=\left(\prod_{r=1}^{s+k}x_{n+r}\right)
\E_k[n,n+s+k-2].
\end{equation}
\end{lemma}
\begin{proof}
For $k=0$ this is \eqref{eq:Bzero}. If $s=k\geq1$, the only independent
set of size $k$ in a path of $2k-1$ vertices is
$\{n,n+2,\ldots,n+2k-2\}$. Cancellation using \eqref{eq:u} gives
\eqref{eq:Bdiag}. For $s>k$, multiply \eqref{eq:last} by
$\prod_{r=1}^{s+k}x_{n+r}$. Its two terms give respectively
$x_{n+s+k}B_{s-1}^k(n)$ and
$x_{n+s+k-2}x_{n+s+k+1}B_{s-1}^{k-1}(n)$.
Thus \eqref{eq:Brec} holds.
\end{proof}

For each integer $g\geq1$, put $N=2g+2$ and write
$\boldsymbol\alpha=(\alpha_0,\ldots,\alpha_g)$ for the ordered parameter
tuple. Define the following Laurent polynomial in its listed $N-1$ arguments:
\begin{equation}\label{eq:L}
L_N(x_1,\ldots,x_{N-1};\boldsymbol\alpha)
=x_2x_{N-2}\sum_{k=0}^{g-1}\alpha_k\E_k[1,N-4]
+\alpha_gx_1x_{N-1}.
\end{equation}
For $N=4$ this reads $\alpha_0x_2^2+\alpha_1x_1x_3$.
Without displayed arguments, $L_N$ denotes
$L_N(x_1,\ldots,x_{N-1};\boldsymbol\alpha)$; translated argument lists
will be written explicitly or obtained by the shift defined below.
Lemma \ref{lem:B} rewrites \eqref{eq:RN} as
\begin{equation}\label{eq:RNL}
t_nt_{n+N}=L_N(t_{n+1},\ldots,t_{n+N-1};\boldsymbol\alpha).
\end{equation}
The next identity describes the possible internal poles in \eqref{eq:L}.
For $3\leq r\leq N-3$, set
\begin{equation}\label{eq:H}
H_r=x_{r-2}x_{r+1}^2+x_{r-1}^2x_{r+2}.
\end{equation}
\begin{lemma}\label{lem:poles}
For $N\geq6$, the only possible denominators of $L_N$ are products of
distinct variables among $x_3,\ldots,x_{N-3}$. For every integer $k\geq0$
and $3\leq r\leq N-3$, multiply $\E_k[1,N-4]$ by $x_r$ to clear its
possible pole at $x_r=0$. Evaluation at zero then gives
\begin{equation}\label{eq:residue}
\left.x_r\E_k[1,N-4]\right|_{x_r=0}
=\frac{H_r}{x_{r-1}x_{r+1}}
\sum_{a=0}^{k-1}\E_a[1,r-4]\E_{k-1-a}[r+1,N-4].
\end{equation}
\end{lemma}
\begin{proof}
The two weights having $x_r$ in their denominators are $u_{r-2}$ and
$u_{r-1}$. They are adjacent, so a summand in \eqref{eq:E} contains at
most one of them. The weights $u_{r-3}$ and $u_r$, when present, have
$x_r$ in their numerators. Terms containing either of these weights do
not contribute after multiplication by $x_r$ and specialization at zero.
The surviving terms select $r-2$ or $r-1$ and an independent set on the
two outside paths $[1,r-4]$ and $[r+1,N-4]$. Distributing their $k-1$
selected vertices gives the sum on the right of \eqref{eq:residue}.
The sum of the two leading coefficients is
\[
\frac{x_{r-2}x_{r+1}}{x_{r-1}}
+\frac{x_{r-1}x_{r+2}}{x_{r+1}}
=\frac{H_r}{x_{r-1}x_{r+1}},
\]
which proves \eqref{eq:residue}, including the truncated paths at both
boundaries. The same adjacent-vertex argument bounds every pole by one.
At $x_2$ only $u_1$ can contribute a pole, and at $x_{N-2}$ only
$u_{N-4}$ can do so. The prefactor $x_2x_{N-2}$ in \eqref{eq:L} cancels
these poles. No other variable can occur in a denominator.
\end{proof}

For $N\geq6$, the expression $L_{N-2}$ is defined by \eqref{eq:L}
with $g$ replaced by $g-1$ and parameter tuple
$(\alpha_1,\ldots,\alpha_g)$. The variable arguments are assigned in
their listed order. Its three translates used at the boundary are
\begin{align}
q_L&=L_{N-2}(x_1,\ldots,x_{N-3};\alpha_1,\ldots,\alpha_g),\notag\\
q_M&=L_{N-2}(x_2,\ldots,x_{N-2};\alpha_1,\ldots,\alpha_g),\label{eq:q}\\
q_R&=L_{N-2}(x_3,\ldots,x_{N-1};\alpha_1,\ldots,\alpha_g).\notag
\end{align}
For instance,
\begin{equation}\label{eq:qM}
q_M=x_3x_{N-3}\sum_{k=0}^{g-2}\alpha_{k+1}\E_k[2,N-5]
+\alpha_gx_2x_{N-2}.
\end{equation}
Let $Z_i=L_N|_{x_i=0}$ for $i\in\{1,2,N-2,N-1\}$;
these specializations are defined by Lemma \ref{lem:poles}.
Define the shift $\tau$ on
Laurent polynomials by $\tau(x_i)=x_{i+1}$ and
$\tau(\alpha_k)=\alpha_k$, extended by addition and multiplication.
For example, $\tau u_i=u_{i+1}$.

\begin{lemma}\label{lem:boundary}
For $N=2g+2\geq6$, the following identities hold:
\begin{align}
L_N&=Z_1+\frac{x_1}{x_3}q_R
     =Z_{N-1}+\frac{x_{N-1}}{x_{N-3}}q_L,\label{eq:decomp}\\
Z_2&=\frac{x_1}{x_3}q_R,
\qquad Z_{N-2}=\frac{x_{N-1}}{x_{N-3}}q_L,\label{eq:Z2}\\
\tau L_N&=\frac{x_3x_{N-1}}{x_2x_{N-2}}Z_1
          +\frac{x_N}{x_{N-2}}q_M,\label{eq:tauL}\\
L_N(x_0,\ldots,x_{N-2};\boldsymbol\alpha)
&=\frac{x_1x_{N-3}}{x_2x_{N-2}}Z_{N-1}
  +\frac{x_0}{x_2}q_M.\label{eq:leftL}
\end{align}
\end{lemma}
\begin{proof}
Apply \eqref{eq:first} to \eqref{eq:L}. The part omitting vertex $1$ is
\[
Z_1=x_2x_{N-2}\sum_{k=0}^{g-1}\alpha_k\E_k[2,N-4].
\]
The remaining part, together with $\alpha_gx_1x_{N-1}$, is
$x_1q_R/x_3$. This proves the first decomposition. In $Z_1$, the only
weight containing $x_2$ is $u_2$, which has $x_2$ in its numerator;
therefore $Z_1|_{x_2=0}=0$. Since $q_R$ is independent of $x_2$, this
also proves the first identity in \eqref{eq:Z2}. Reversing the ordered
arguments gives the other two identities.

For \eqref{eq:tauL}, separate the last vertex $N-3$ in the shifted
path $[2,N-3]$. Its weight is
$u_{N-3}=x_{N-3}x_N/(x_{N-2}x_{N-1})$. The terms not selecting it give
the first term of \eqref{eq:tauL}; the terms selecting it, together with
$\alpha_gx_2x_N$, give $x_Nq_M/x_{N-2}$ by \eqref{eq:qM}.
Finally, split the first vertex in the path $[0,N-5]$ for the left-shifted
expression. Its constant part in $x_0$ is
$x_1x_{N-3}\sum_{k=0}^{g-1}\alpha_k\E_k[1,N-5]$.
The last-vertex decomposition of $L_N$ identifies this as the first term
of \eqref{eq:leftL}, and the remaining term is $x_0q_M/x_2$.
\end{proof}

\section{Seed construction}

We use the definitions of \cite[Section 2]{LP}. Let $S$ be a unique
factorization domain (UFD), and let $\Frac(S)$ denote its fraction field.
An LP seed consists of a transcendence basis $x_0,\ldots,x_{N-1}$ of
$\Frac(S)(x_0,\ldots,x_{N-1})$, together with irreducible
polynomials $F_i\in S[x_0,\ldots,x_{N-1}]$, $0\leq i\leq N-1$, with no variable factors,
where $F_i$ does not involve $x_i$.
Its exchange Laurent polynomial $\wh F_i$ is obtained from $F_i$ by
dividing by a monomial in the variables $x_j$, $j\ne i$, with
nonnegative exponents. In the formula below, a hat over a ring generator
means that it is omitted. For each $j\ne i$, introduce a new indeterminate
$z$ and require that
\begin{equation}\label{eq:LPnorm}
\left.\wh F_i\right|_{x_j=F_j/z}
\in S[x_0^{\pm1},\ldots,\widehat{x_j^{\pm1}},\ldots,x_{N-1}^{\pm1},z^{\pm1}]
\quad\text{is not divisible by }F_j.
\end{equation}
Divisibility in \eqref{eq:LPnorm} is taken in that Laurent polynomial ring.
Mutation at $x_0$ replaces it by $x_N=\wh F_0/x_0$; write $F'_i$,
$0\leq i\leq N-1$, for the exchange polynomials of the mutated seed.
In applying the identities of Section 2, the formal variable $x_N$ is
evaluated at this rational function, and index shifts are performed before
this evaluation. Set $F'_0=F_0$. For $i\ne0$, set $F'_i=F_i$ when $F_i$
is independent of $x_0$. Otherwise, \cite[Lemma 2.7]{LP} implies that
$x_i$ does not divide $F_0/\wh F_0$, so $\wh F_0|_{x_i=0}$ is defined.
It is nonzero because $F_0$ has no variable factor. Form
\begin{equation}\label{eq:LPmut}
G_i=\left.F_i\right|_{x_0=(\wh F_0|_{x_i=0})/x_N}.
\end{equation}
Remove from $G_i$ all factors shared with $\wh F_0|_{x_i=0}$, including
their full multiplicities in $G_i$, and multiply by a Laurent monomial
to obtain $F'_i$, a polynomial without variable factors. This determines
$F'_i$ up to a unit of $S$, namely $\pm1$ for the ring below.
The factor removals can be checked after inverting the unchanged variables
$x_1,\ldots,x_{N-1}$, leaving the coefficient ring $S$ unchanged.

The Laurent theorem \cite[Theorem 5.1]{LP} states that every cluster
variable obtained by mutations belongs to
$S[x_0^{\pm1},\ldots,x_{N-1}^{\pm1}]$.

Fix $N=2g+2\geq6$ and $S=\Z[\alpha_0,\ldots,\alpha_g]$. Define
\begin{equation}\label{eq:D}
D_0=\prod_{r=3}^{N-3}x_r,\qquad
D_1=\prod_{r=4}^{N-4}x_r,\qquad
D_R=\prod_{r=3}^{N-5}x_r.
\end{equation}
In particular $D_1=D_R=1$ when $N=6$.
Extend \eqref{eq:H} to $2\leq i\leq N-3$. With $q_L,q_M$ as in
\eqref{eq:q}, set
\begin{align}
F_0&=D_0L_N(x_1,\ldots,x_{N-1};\boldsymbol\alpha),\notag\\
F_1&=D_1(x_0x_3x_{N-1}+x_2q_M),\notag\\
F_i&=H_i\qquad(2\leq i\leq N-3),\label{eq:seed}\\
F_{N-2}&=D_R(x_0x_{N-4}x_{N-1}+x_{N-3}q_L),\notag\\
F_{N-1}&=\left(\prod_{r=2}^{N-4}x_r\right)L_N(x_0,\ldots,x_{N-2};\boldsymbol\alpha).\notag
\end{align}

\begin{proposition}\label{prop:seed}
The polynomials \eqref{eq:seed} form an LP seed over $S$. Its exchange
Laurent polynomial at $x_0$ is
\begin{equation}\label{eq:hat0}
\wh F_0=L_N.
\end{equation}
\end{proposition}
\begin{proof}
Lemma \ref{lem:poles}, also applied at order $N-2$ with shifted variables
and parameters when $N\geq8$, proves that all expressions in
\eqref{eq:seed} are polynomials. For $N=6$, the lower-order expressions
are copies of the polynomial $L_4$. Each expression omits its assigned
variable.

The coefficient of $\alpha_0$ in $F_0$ is the monomial $D_0x_2x_{N-2}$.
Neither $x_2$ nor $x_{N-2}$ divides the constant part in $\alpha_0$,
because of the term $\alpha_gD_0x_1x_{N-1}$. For an internal variable
$x_r$ with $3\leq r\leq N-3$, the coefficient of $\alpha_1$ in this
constant part has nonzero specialization at $x_r=0$: by
\eqref{eq:residue} for $k=1$, it specializes to
\[
\frac{D_0}{x_r}\,x_2x_{N-2}
\frac{H_r}{x_{r-1}x_{r+1}}\ne0.
\]
This equality may be checked in the Laurent ring of the other variables;
its left side is a polynomial. Thus the leading coefficient and the
constant part in $\alpha_0$ are relatively prime in the polynomial ring
\[
\Z[\alpha_1,\ldots,\alpha_g,x_1,\ldots,x_{N-1}].
\]
Gauss's lemma and linearity in $\alpha_0$ prove that $F_0$ is
irreducible. The same observations exclude every variable factor:
$x_1,x_{N-1}$ do not divide the $\alpha_0$ term, and $x_0$ is absent.

Next regard $F_1$ as linear in $x_0$. Its leading coefficient is the
monomial $D_1x_3x_{N-1}$, with integer coefficient $1$.
The constant part $x_2D_1q_M$ is not divisible by $x_3$, since its
$\alpha_g$ term is $\alpha_gx_2^2D_1x_{N-2}$. It is independent of
$x_{N-1}$. For $4\leq r\leq N-4$, the lower-order version of
\eqref{eq:residue}, using the coefficient $\alpha_2$ in $q_M$, shows
that its specialization at $x_r=0$ is nonzero. This range is empty for
$N=6$; whenever it is nonempty, $g\geq3$ and $\alpha_2$ is present.
Hence the two coefficients in $x_0$ are relatively prime in
$S[x_1,\ldots,x_{N-1}]$, so $F_1$ is
irreducible and has no variable factor.

Reversing the variables by $x_j\mapsto x_{N-1-j}$ exchanges $F_0$ with
$F_{N-1}$ and $F_1$ with $F_{N-2}$. Indeed it reverses the path in
\eqref{eq:L} and exchanges $q_M$ with $q_L$. This proves the same
claims for the two right boundary polynomials. Finally, $H_i$ is linear
in $x_{i-2}$, with relatively prime coefficients $x_{i+1}^2$ and
$x_{i-1}^2x_{i+2}$. Thus all seed axioms hold.

For normalization, \cite[Lemma 2.7]{LP} says that a variable $x_j$ can
divide $F_0/\wh F_0$ only if $F_j$ is independent of $x_0$.
The polynomials $F_1,H_2,F_{N-2},F_{N-1}$ all depend on $x_0$.
The remaining candidates are exactly $j=3,\ldots,N-3$, for which
$F_j=H_j$. Formula \eqref{eq:residue} implies
\[
H_j\mid \left.F_0\right|_{x_j=0}
\]
in the Laurent ring of the remaining variables, and therefore
$H_j\mid F_0|_{x_j=H_j/z}$. The coefficient of $\alpha_0$ in this last
expression is
\[
\frac{H_j}{z}\,\frac{D_0}{x_j}x_2x_{N-2}.
\]
The remaining factor $(D_0/x_j)x_2x_{N-2}/z$ is a Laurent monomial
unit. Since $H_j$ is irreducible and has no variable factor, it remains
a nonunit after this localization. The displayed coefficient is therefore
divisible by $H_j$ exactly once. As $H_j$ is independent of $\alpha_0$,
divisibility of $F_0|_{x_j=H_j/z}$ by $H_j^2$ would force the same
divisibility of every coefficient in $\alpha_0$, a contradiction.
Consequently the exponent of $x_j$ in $F_0/\wh F_0$ is exactly one.
These computations give $F_0/\wh F_0=D_0$, proving \eqref{eq:hat0}.
\end{proof}

\section{Mutation and the Laurent property}

Only four exchange polynomials other than $F_0$ depend on $x_0$.
The boundary identities make their mutation explicit.

\begin{proposition}\label{prop:period}
Mutating the seed \eqref{eq:seed} at $x_0$, and choosing positive signs
for the exchange polynomials, gives
\begin{equation}\label{eq:period}
F'_0=\tau F_{N-1},\qquad F'_i=\tau F_{i-1}\quad(1\leq i\leq N-1).
\end{equation}
Thus the mutated seed, written in the ordered window
$(x_1,\ldots,x_{N-1},x_N)$ with $x_N=L_N/x_0$, has the same form
as the original seed.
\end{proposition}
\begin{proof}
For $i\in\{1,2,N-2,N-1\}$, Proposition \ref{prop:seed} gives
$\wh F_0|_{x_i=0}=Z_i$, with $Z_i$ as in Lemma \ref{lem:boundary}.
The boundary identities give
\begin{align}
\left.\frac{F_1}{D_1}\right|_{x_0=Z_1/x_N}
 &=\frac{x_2x_{N-2}}{x_N}\,\tau L_N,\label{eq:mut1}\\
\left.H_2\right|_{x_0=Z_2/x_N}
 &=\frac{x_1}{x_N}\,\tau\left(\frac{F_1}{D_1}\right),\label{eq:mut2}\\
\left.\frac{F_{N-2}}{D_R}\right|_{x_0=Z_{N-2}/x_N}
 &=\frac{q_L}{x_{N-3}x_N}\,\tau H_{N-3},\label{eq:mutR}\\
\left.\frac{F_{N-1}}{\prod_{r=2}^{N-4}x_r}\right|_{x_0=Z_{N-1}/x_N}
 &=\frac{Z_{N-1}}{x_2x_{N-2}x_N}
       \,\tau\left(\frac{F_{N-2}}{D_R}\right).\label{eq:mutlast}
\end{align}
Equation \eqref{eq:mut1} is \eqref{eq:tauL} multiplied by
$x_2x_{N-2}/x_N$. In \eqref{eq:mut2}, the left side is
$x_1(x_3q_R+x_1x_4x_N)/x_N$, and the parenthesis is
$\tau(F_1/D_1)$. Formula \eqref{eq:mutR} follows from
$Z_{N-2}=x_{N-1}q_L/x_{N-3}$ and
\[
\tau H_{N-3}=x_{N-4}x_{N-1}^2+x_{N-3}^2x_N.
\]
Finally \eqref{eq:leftL} gives \eqref{eq:mutlast}, since
\[
\tau(F_{N-2}/D_R)=x_1x_{N-3}x_N+x_{N-2}q_M.
\]

Restoring the monomial prefactors in \eqref{eq:seed}, the expressions
$G_i$ are variable-monomial multiples of $\tau F_{i-1}$, with the
following additional factors:
\[
\begin{array}{c|cccc}
i&1&2&N-2&N-1\\\hline
\text{additional factor}&1&1&q_L&Z_{N-1}.
\end{array}
\]
All denominators of these factors are variable monomials.
Every nonvariable factor of $q_L$ occurs in
$Z_{N-2}=x_{N-1}q_L/x_{N-3}$, and the last factor is $Z_{N-1}$ itself.
These are exactly the factors that \eqref{eq:LPmut} deletes.
For each of these four indices, the target $\tau F_{i-1}$ remains
irreducible in $S[x_1^{\pm1},\ldots,x_{N-1}^{\pm1}][x_N]$, since it
is irreducible before localization and has no variable factor. A nonunit
common factor would therefore force $\tau F_{i-1}$ to divide $Z_i$.
This is impossible: $\tau F_{i-1}$ is nonconstant in $x_N$, whereas
the nonzero element $Z_i$ lies in the coefficient ring
$S[x_1^{\pm1},\ldots,x_{N-1}^{\pm1}]$. Further inverting $x_N$ cannot
create a common factor, since $x_N$ does not divide the target.
Polynomial normalization therefore yields
$F'_i=\tau F_{i-1}$ for the four indices in the table. 

For $3\leq i\leq N-3$, $F_i=H_i$ is independent of $x_0$, and
$H_i=\tau H_{i-1}$. These polynomials are unchanged by mutation.
At the mutation position,
$F'_0=F_0=\tau F_{N-1}$, directly from \eqref{eq:seed}.
This proves \eqref{eq:period} for every index.
\end{proof}

\begin{proof}[Proof of Theorem \ref{thm:main}]
For $N\geq6$, Propositions \ref{prop:seed} and \ref{prop:period} identify
each forward step of \eqref{eq:RNL} with a mutation of the seed
\eqref{eq:seed}, followed by cyclic relabeling. The LP Laurent theorem
over $S=\Z[\alpha_0,\ldots,\alpha_g]$ gives the asserted membership
for $n\geq0$. For $N=4$, the same membership is
\cite[Theorem 1.8]{FZ}, with $(p,q,N)=(1,2,4)$ and coefficients
$(\alpha_1,\alpha_0)$.

The inverse of the window map in \eqref{eq:RNL} is
\[
t_{n-1}=
\frac{L_N(t_n,\ldots,t_{n+N-2};\boldsymbol\alpha)}{t_{n+N-1}}.
\]
The forward and inverse window maps are mutually inverse rational maps.
Hence every window is a transcendence basis over
$\Q(\alpha_0,\ldots,\alpha_g)$, and its entries
are nonzero elements of the rational function field.
Finally, reversing the arguments $x_i\mapsto x_{N-i}$ in \eqref{eq:L}
reverses the weights $u_i\mapsto u_{N-3-i}$ and preserves every
independent-set sum. Hence $L_N$ is invariant under reversal.
The reversed sequence $(t_{N-1-n})_{n\in\Z}$ satisfies the same recurrence.
Applying the forward result with initial values $t_{N-1},\ldots,t_0$
proves the assertion for negative indices.
\end{proof}

\begin{corollary}\label{cor:integers}
Assign integer values to the parameters and values in $\{1,-1\}$ to
the initial variables. Evaluation of the universal Laurent polynomials
gives an integer sequence satisfying the polynomial identities
\eqref{eq:RN}. It agrees with the recursively generated sequence
wherever all required denominators are nonzero.
\end{corollary}
\begin{proof}
The evaluation map on the Laurent ring of Theorem \ref{thm:main} is
defined for these initial values and takes values in $\Z$. Apply it
to \eqref{eq:RN}.
\end{proof}

\section{The order-six recurrence}

For $N=6$, the seed \eqref{eq:seed} consists of
\begin{align*}
F_0&=\alpha_0x_2x_3x_4+\alpha_1x_1x_4^2
     +\alpha_1x_2^2x_5+\alpha_2x_1x_3x_5,\\
F_1&=\alpha_1x_2x_3^2+\alpha_2x_2^2x_4+x_0x_3x_5,\\
F_2&=x_0x_3^2+x_1^2x_4,\\
F_3&=x_1x_4^2+x_2^2x_5,\\
F_4&=\alpha_1x_2^2x_3+\alpha_2x_1x_3^2+x_0x_2x_5,\\
F_5&=\alpha_0x_1x_2x_3+\alpha_1x_0x_3^2
     +\alpha_1x_1^2x_4+\alpha_2x_0x_2x_4.
\end{align*}
The normalizations at the two ends are
\[
\wh F_0=F_0/x_3,\qquad \wh F_5=F_5/x_2;
\]
the other four exchange polynomials have trivial normalization.
For the first identity, write
\[
F_0=x_3(\alpha_0x_2x_4+\alpha_2x_1x_5)+\alpha_1F_3.
\]
Substitution gives
\[
\left.F_0\right|_{x_3=F_3/z}
=F_3\left(\frac{\alpha_0x_2x_4+\alpha_2x_1x_5}{z}+\alpha_1\right).
\]
The second factor has $\alpha_1$-coefficient $1$, so it is not divisible
by $F_3$. Thus $F_3$ divides the substituted expression exactly once.
Reflection gives the normalization at the other end. For $1\leq i\leq4$, inspection of the
six displayed polynomials shows that every $F_j$ with $j\ne i$
depends on $x_i$; \cite[Lemma 2.7]{LP} therefore excludes a normalization
factor in these four directions.
The exchange $x_0x_6=\wh F_0$ is exactly \eqref{eq:R6}.

The four substitutions in \eqref{eq:LPmut} are explicit in this case:
\[
\renewcommand{\arraystretch}{1.65}
\begin{array}{c|c|c}
i&Z_i&G_i\\\hline
1&\dfrac{x_2(\alpha_0x_3x_4+\alpha_1x_2x_5)}{x_3}
 &(x_2/x_6)\tau F_0\\
2&\dfrac{x_1(\alpha_1x_4^2+\alpha_2x_3x_5)}{x_3}
 &(x_1/x_6)\tau F_1\\
4&\dfrac{x_5(\alpha_1x_2^2+\alpha_2x_1x_3)}{x_3}
 &\dfrac{\alpha_1x_2^2+\alpha_2x_1x_3}{x_3x_6}\tau F_3\\
5&\dfrac{x_4(\alpha_0x_2x_3+\alpha_1x_1x_4)}{x_3}
 &\dfrac{\alpha_0x_2x_3+\alpha_1x_1x_4}{x_3x_6}\tau F_4.
\end{array}
\]
The last two rows delete the factors
$\alpha_1x_2^2+\alpha_2x_1x_3$ and
$\alpha_0x_2x_3+\alpha_1x_1x_4$, respectively.
Together with $F_3=\tau F_2$ and $F_0=\tau F_5$, this explicitly
exhibits the period-one mutation of the $R_6$ seed.

\bibliographystyle{plain}
\bibliography{references}

\bigskip
\noindent\textsc{Zhili College, Tsinghua University, Beijing 100084, China}

\noindent\textit{Email address:}
\href{mailto:tangqy24@mails.tsinghua.edu.cn}{tangqy24@mails.tsinghua.edu.cn}
\end{document}